\documentclass[preprint]{elsarticle}

\pdfoutput=1

\makeatletter
\def\ps@pprintTitle{%
 \let\@oddhead\@empty
 \let\@evenhead\@empty
 \def\@oddfoot{\centerline{\thepage}}%
 \let\@evenfoot\@oddfoot}
\makeatother

\usepackage[utf8]{inputenc}
\usepackage[T1]{fontenc}
\usepackage[hidelinks]{hyperref}
\usepackage{microtype}

\usepackage{fancyvrb} 

\usepackage{mathtools, amsthm, amsfonts, amssymb}
\usepackage{booktabs}
\usepackage{commath}
\usepackage{caption}

\newtheorem{theorem}{Theorem}
\newtheorem{proposition}[theorem]{Proposition}
\newtheorem{lemma}[theorem]{Lemma}

\theoremstyle{definition}
\newtheorem{definition}[theorem]{Definition}

\newtheorem*{problem*}{Problem}
\newtheorem{problem}[theorem]{Problem}

\newtheorem{example}[theorem]{Example}

\DeclarePairedDelimiter\floor{\lfloor}{\rfloor}
\DeclareMathOperator{\rank}{R}
\DeclareMathOperator{\borderrank}{\underline{R}}
\DeclareMathOperator{\rad}{rad}

\newcommand{\CC}{\mathbb{C}}
\newcommand{\extbinom}[3]{\binom{#1}{#2}_{\!\!#3}}

\begin{document}

\begin{frontmatter}
\title{A note on the gap between rank and border rank}
\author{Jeroen Zuiddam}
\ead{j.zuiddam@cwi.nl}
\address{Centrum Wiskunde \& Informatica, Science Park 123, Amsterdam}

\begin{abstract}
We study the tensor rank of the tensor corresponding to the algebra of $n$-variate complex polynomials modulo the $d$th power of each variable.
As a result we find a sequence of tensors with a large gap between rank and border rank, and thus a counterexample to a conjecture of Rhodes. At the same time we obtain a new lower bound on the tensor rank of tensor powers of the generalised W-state tensor. In addition, we  exactly determine the tensor rank of the tensor cube of the three-party W-state tensor, thus answering a question of Chen et al.
\end{abstract}

\end{frontmatter}

\section{Introduction}

Let $V_1,\ldots,V_k$ be finite-dimensional complex vector spaces and consider the vector space $V\coloneqq V_1 \otimes \cdots \otimes V_k$  of $k$-tensors. A tensor of the form $v_1\otimes v_2 \otimes \cdots \otimes v_k$ in $V$ is called \emph{simple}. The \emph{tensor rank} of a tensor~$t\in V$ is the smallest number~$r$ such that $t$ can be written as a sum of $r$ simple tensors. The \emph{border rank}~$\borderrank(t)$ of~$t$ is the smallest number $r$ such that $t$ is the limit of a sequence of tensors in $V$ of rank at most $r$, in the Euclidean topology. Clearly, $\borderrank(t) \leq \rank(t)$ and already in the small space $\CC^2 \otimes \CC^2 \otimes \CC^2$ there exist tensors with $\borderrank(t) < \rank(t)$.

Tensor rank plays a fundamental role in various problems in modern applied mathematics. One famous example is the problem of deciding the complexity of matrix multiplication \cite{burgisser1997algebraic}. We refer to \cite{jm2012tensors} for more examples of applications of tensor rank. While the tensor rank of a 2-tensor (matrix rank) can be efficiently computed, computing the tensor rank of a $k$-tensor  is NP-hard when $k\geq 3$ \cite{haastad1990tensor,hillar2013most,shitov2016hard}. 
The border rank notion is important for at least the following two reasons. Unlike tensor rank, border rank is defined by polynomial equations. One approach for computing a lower bound for the tensor rank of a tensor is thus to find the relevant border rank equation and then verify that the tensor does not satisfy the equation. This strategy was for example used in \cite{hauenstein2013equations}. On the other hand, border rank upper bounds can in some situations be turned into nontrivial rank upper bounds, especially when one is interested in the asymptotic behaviour of tensor rank when taking large tensor powers of a fixed tensor. This idea is crucial in, for example, the laser method of Strassen~\cite{MR882307} and all later improvements of this method, see for example~\cite{le2014powers}.

This paper is motivated by the following basic question about tensor rank and border rank.

\begin{problem}
What is the maximal ratio $\rank(t)/\borderrank(t)$ for a $k$-tensor $t$ in $(\CC^n)^{\otimes k}$?
\end{problem}

Our main result is the following lower bound. Let $e_0, e_1$ be the standard basis of $\CC^2$. Define $W_k$ to be the tensor $e_1\otimes e_0 \otimes \cdots \otimes e_0\, +\, e_0 \otimes e_1\otimes \cdots \otimes e_0\, +\, \cdots\, +\, e_0 \otimes \cdots \otimes e_0 \otimes e_1$ living in $(\CC^2)^{\otimes k}$. This tensor is known as the generalised W-state tensor in quantum information theory.

\begin{theorem}\label{mainthintro}
Let $k\geq 3$. There exists an explicit sequence of $k$-tensors $t_n$ in $(\CC^{2^n})^{\otimes k}$ such that
\[
\frac{\rank(t_{n})}{\borderrank(t_{n})} \geq k - o(1),
\]
when $n$ goes to infinity. Namely, let $t_n = W_k^{\otimes n} \in (\CC^{2^n})^{\otimes k}\!,$ the $n$-fold tensor Kronecker product of $W_k$. Then $\borderrank(t_n) = 2^n$ and $\rank(t_n) \geq k\cdot 2^n - o(2^n)$, when~$n$ goes to infinity.
\end{theorem}

We obtain Theorem\nobreakspace \ref {mainthintro} by applying a tensor rank lower bound of Bläser to the tensor corresponding to the algebra $A_{d,n} \coloneqq\CC[x_1,\ldots,x_n]/(x_1^d,\ldots,x_n^d)$ of $n$-variate complex polynomials modulo the $d$th power of each variable.
This in turn leads to the aforementioned lower bound on the tensor rank of tensor powers of the generalised W-state tensor $W_k$. Our bound improves the lower bound $\rank(W_k^{\otimes n}) \geq (k-1)\cdot 2^n - k + 2$ of Chen et al.~\cite{chen2010tensor}.

We note that it is a major open problem to find explicit tensors $t \in (\CC^n)^{\otimes 3}$ with $\rank(t) \geq (3+\varepsilon)n$ for any $\varepsilon >0$ \cite{blaser2014explicit}. There are explicit tensors $t \in (\CC^n)^{\otimes 3}$ known with $\rank(t) \geq (3-o(1)) n$ when $n$ goes to infinity, see \cite[Theorem~2]{MR2065853}.

\paragraph{Related work}
De Silva and Lim show that for a 3-tensor $t$ the difference between tensor rank and border rank $\rank(t)-\borderrank(t)$ can be arbitrarily large \cite{de2008tensor}. However, their result implies a lower bound of only $3/2$ on the maximal ratio $\rank(t)/\borderrank(t)$ for $t$ a 3-tensor.

Allman et al.\ give explicit tensors $K_n$ in $\CC^n\otimes \CC^n\otimes \CC^n$ of border rank~$n$ and rank $2n-1$ \cite{allman2013tensor}; a rank to border rank ratio that converges to~2. They provide references to other tensors with similar rank and border rank behaviour. We note that the tensor $K_n$ is essentially the tensor of the algebra $\CC[x]/(x^n)$. It was conjectured by Rhodes that the rank of a tensor in~$\CC^n\otimes \CC^n \otimes \CC^n$ of border rank~$n$ is at most $2n-1$ {\cite[Conjecture~0]{ballico2013stratification}}. Theorem\nobreakspace \ref {mainthintro} shows that this conjecture is false.

Independently of the author and with different techniques, Landsberg and Michałek have recently constructed a sequence of 3-tensors 
 with a ratio of rank to border rank converging to $5/2$, thus also disproving the above conjecture~\cite{2015arXiv150403732L}.

As is also mentioned in \cite{2015arXiv150403732L}, we note that for any $k\geq 3$, the tensor $W_k\in (\CC^2)^{\otimes k}$ has border rank $2$ and rank $k$, thus giving a rank to border rank ratio of~$k/2$, see the proof of Theorem\nobreakspace \ref {mainthintro}.

As pointed out by an anonymous reviewer, a lower bound on the maximal ratio between rank and border rank for 3-tensors ($k=3$) similar to the one in Theorem\nobreakspace \ref {mainthintro} in this paper, can also be obtained as follows. Let $I_d(x_1, \ldots, x_n) \subseteq \CC[x_1, \ldots, x_n]$ be the ideal generated by all monomials of degree~$d$. Bläser~\cite{MR2065853} proves that the tensor corresponding to the algebra $P_{n,d}$ defined as  $\CC[x_1, \ldots, x_n]/I_{d+1}(x_1, \ldots, x_n)$ has tensor rank $\rank(P_{n,d}) \geq (3-o(1)) \dim(P_{n,d})$. Moreover, the ideal $I_{d+1}(x_1, \ldots, x_n)$ is a so-called monomial ideal and is therefore ``smoothable''. It turns out (see \cite{blser_et_al:LIPIcs:2016:6434}) that associative unital algebras defined by smoothable ideals (like $P_{n,d}$) have ``minimal border rank'' which in this case means that $\borderrank(P_{n,d}) = \dim P_{n,d}$. Combining these observations yields, for any $d>1$, an explicit sequence of 3-tensors $t_n  \in (\CC^{\dim(P_{n,d})})^{\otimes 3}$ such that $\rank(t_n)/ \borderrank(t_n) \geq 3 - o(1)$ when $n$ goes to infinity. Note that the algebra $P_{n,d}$ is slightly different from the algebra $A_{d,n}$ that we study here.

Very little is known about general upper bounds on the rank to border rank ratio. We are only aware of the following bound that can be deduced from a result of Lehmkuhl and Lickteig \cite{lehmkuhl1989order} and Proposition 15.26 in \cite{burgisser1997algebraic}. For any tensor $t\in \CC^n \otimes \CC^n \otimes \CC^n$ we have $\rank(t)/\borderrank(t) \leq 2 \cdot 9^{(n-1)\borderrank(t)} + 1$.

\paragraph{Outline} This paper is organised as follows. First we introduce the algebra $A_{d,n}$ and for the corresponding tensor study its border rank and tensor rank. Then we observe that this tensor specialises to powers of the W-state tensor, yielding the gap between rank and border rank given above. Finally, we compute the tensor rank of the tensor cube of the three-party W-state tensor.

\section{The algebra $A_{d,n}$}
Many examples of interesting 3-tensors come from algebras. A complex algebra is a complex vector space $V$ together with a \emph{multiplication} defined by a bilinear map $\phi:V \times V \to V$. An algebra is called \emph{associative} if $\phi(\phi(u,v),w) = \phi(u,\phi(v,w))$ for all $u,v,w\in V$. An algebra is called \emph{unital} if there is an element $e\in V$ such that $\phi(e,v) = \phi(v,e) = v$ for all $v\in V$.
Let $e_1, e_2, \ldots$ be a basis of~$V$ and $e_1^*, e_2^*, \ldots$ the dual basis. We can naturally view the algebra $(V,\phi)$ as a tensor in $V\otimes V\otimes V$ by
\[
\phi \mapsto \sum_{i,j,k} e_k^*(\phi(e_i,e_j))\, e_i\otimes e_j\otimes e_k,
\]
called the \emph{structure tensor}. In this way we can speak about the tensor rank and border rank of an algebra.  There are many results on the tensor rank and border rank of algebras, in particular of the algebra of $n\times n$ matrices, for which we refer to \cite{burgisser1997algebraic} and \cite{blaser2000lower}. For results on the tensor rank and border rank of general tensors we refer to \cite{jm2012tensors}. In this section we will study the complex associative unital algebra
\[
A_{d,n} \coloneqq (\CC[x]/(x^d))^{\otimes n} = \CC[x_1,\ldots,x_n]/(x_1^d,\ldots,x_n^d),
\]
of $n$-variate complex polynomials modulo the $d$th power of each variable.

\subsection{Border rank}
 A tensor $t$ in $V_1\otimes \cdots \otimes V_k$ is called 1-\emph{concise} if there does not exist a proper subspace $U_1 \subseteq V_1$ such that $t \in U_1 \otimes V_2 \otimes\cdots \otimes V_k$. Similarly, we define $i$-conciseness for $i\in \{2,\ldots,k\}$. A tensor is called \emph{concise} if it is $i$-concise for all $i$. We can think of a concise tensor as a tensor that ``uses'' all dimensions of the local spaces $V_i$. Tensors of unital algebras are concise. For a concise tensor~$t$ in $V_1\otimes\cdots\otimes V_k$ the border rank is at least $\max_i \dim V_i$ \cite[Lemma 15.23]{burgisser1997algebraic}. The following proposition is a direct consequence of the well-known fact that $\borderrank(\CC[x]/(x^d)) =  d$, see \cite[Example 15.20]{burgisser1997algebraic}.

\begin{proposition}\label{borderrank}
$\borderrank(A_{d,n})=d^n$. 
\end{proposition}
\begin{proof}
The algebra $A_{d,n}$ is unital. Therefore, the corresponding tensor $A_{d,n} \in \CC^{d^n}\otimes\CC^{d^n}\otimes\CC^{d^n}$ is concise. This implies that $\borderrank(A_{d,n}) \geq d^n$. On the other hand, border rank is submultiplicative under tensor products, so $\borderrank(A_{d,n}) = \borderrank((\CC[x]/(x^d))^{\otimes n}) \leq \borderrank(\CC[x]/(x^d))^n = d^n$. 
\end{proof}

\subsection{Rank lower bound}
Let $(V, \phi)$ be a complex finite-dimensional associative unital algebra. A subspace $I\subseteq V$ is called a \emph{left-ideal} if $\phi(V,I)=I$. A left-ideal $I$ is called \emph{nilpotent} if $I^n=\{0\}$ for some positive integer $n$. The \emph{nilradical} of $V$ is the sum of all nilpotent left-ideals in $A$.

\begin{theorem}[{\cite[Theorem 7.4]{blaser2000lower}}]\label{thblaser} Let $A$ be a finite-dimensional complex associative unital algebra and let $N$ be the nilradical of $A$. For any integer~$m\geq 1$,
\[
\rank(A) \geq \dim(A) - \dim (N^{2m-1}) + 2\dim (N^{m}).
\]
\end{theorem}

We will apply Theorem\nobreakspace \ref {thblaser} to the algebra $A_{d,n}$. Let us first look at a small example.

\begin{example}
Consider the algebra $A\coloneqq A_{2,2} = \CC[x_1,x_2]/(x_1^2,x_2^2)$ of dimension~4. The elements in~$A$ are of the form $\alpha + \beta x_1 + \gamma x_2 + \delta x_1x_2$ with $\alpha,\beta,\gamma,\delta\in \CC$. The nilradical $N\subseteq A$ is the subspace spanned by $\{x_1, x_2,x_1x_2\}$, and hence has dimension 3. The square of the nilradical $N^2$ is spanned by $x_1x_2$ and hence has dimension 1. Taking $m=1$, Theorem\nobreakspace \ref {thblaser} gives $\rank(A) \geq 4 - 3 + 2\cdot 3 = 7$.
\end{example}

We use extended binomial coefficients to get a handle on the dimension of powers of the radical of $A_{d,n}$. Let $\extbinom{n}{b}{\,d}$ be the number of ways to put $b$ balls into $n$ containers with at most $d$ balls per container. This equals the number of monomials of degree~$b$ in $\CC[x_1,\ldots,x_n]/(x_1^{d+1},\ldots,x_n^{d+1})$.

\begin{lemma}\label{asympt}
For any $0\leq q < 1/2$,
\[
\dfrac{\displaystyle\sum_{b=0}^{\floor{qnd}} \extbinom{n}{b}{d}}{\raisebox{-0.2em}{$\displaystyle(d+1)^n$}} \,\to\, 0 \quad\textnormal{ as }\quad n \to \infty.
\]
\end{lemma}
\begin{proof}
Fix $n,d$. The limit
\[
h^d(\rho) \coloneqq \lim_{n\to\infty} \frac1n \ln \extbinom{n}{\rho n}{d}, \quad 0\leq \rho \leq d,
\]
exists, the function $h^d$ is strictly concave, unimodal and reaches its maximum $\ln(m+1)$ at the point~$\rho=d/2$ \cite{fahssi2012polynomial}. Let $0\leq q < 1/2$ and let $\rho\coloneqq qd$. 
For all $\rho < d/2$, $h^d(\rho) < \ln(d+1)$, so there is an $\varepsilon > 0$ such that $h^d(\rho) + \varepsilon < \ln(d+1)$. For $n$ big enough,
\[
\extbinom{n}{\rho n}{d} \leq \exp\bigl((h^d(\rho) + \varepsilon)n\bigr)
\]
and thus
\[
\frac{\sum_{b=0}^{\floor{qnd}} \extbinom{n}{b}{\,d}}{(d+1)^n} \leq (qnd + 1) \exp\bigl(( h^d(\rho)+\varepsilon -\ln(d+1))n\bigr),
\]
which goes to zero as $n$ goes to infinity.

We note that the case $d=1$ can also easily be obtained from the well-known inequality
\[
\sum_{b=0}^{\floor{qn}} \binom{n}{b} \leq 2^{H(q)n},
\]
where $H(q) = -q\log_2 q - (1-q)\log_2(1-q)$ is the binary entropy of $q$, see for example~\cite[Lemma 16.19]{flum2006parameterized}.
\end{proof}

\begin{proposition}\label{maincor}
Let $n\geq 1,d\geq2$ be integers. Then 
\[
\rank(A_{d,n}) \,\geq\, 2\, d^n + \,\,\max_{m\geq 1}\,\Biggl[\,\,\sum_{b=0}^{2m-2} \extbinom{n}{b}{d-1} \!\!-\,\, 2\, \sum_{b=0}^{m-1} \extbinom{n}{b}{d-1} \Biggr] \geq\, 3 d^n - o(d^n).
\]
\end{proposition}
\begin{proof}
The nilradical $N$ of $A_{d,n}$ is the ideal generated by $x_1,\ldots,x_n$, that is, $N$ is the subspace of $A_{d,n}$ of elements with zero constant term. The $m$th power $N^m$ is the subspace spanned by monomials of degree at least $m$, hence the dimension of $N^m$ equals $d^n - \sum_{b=0}^{m-1} \extbinom{n}{b}{\,d-1}$. Theorem\nobreakspace \ref {thblaser} then gives, for any $m \geq 1$,
\begin{align*}
\rank(A_{d,n}) &\geq d^n - \biggl(d^n - \sum_{b=0}^{2m-2} \extbinom{n}{b}{d-1}\biggr) + 2\biggl(d^n - \sum_{b=0}^{m-1} \extbinom{n}{b}{d-1}\biggr)\\
 &=  2d^n + \sum_{b=0}^{2m-2} \extbinom{n}{b}{d-1} -\, 2 \sum_{b=0}^{m-1} \extbinom{n}{b}{d-1}.
\end{align*}
If $2m-1 \leq n(d-1)$, then
\[
\sum_{b=0}^{2m-2}\extbinom{n}{b}{d-1} = d^n - \sum_{b=2m-1}^{n(d-1)}\extbinom{n}{b}{d-1} \!=\, d^n - \sum_{b=0}^{n(d-1)-(2m-1)}\extbinom{n}{b}{d-1},
\]
so
\[
\rank(A_{d,n}) \geq 3\cdot d^n - \sum_{b=0}^{n(d-1)-(2m-1)}\extbinom{n}{b}{d-1} \!-\, 2\sum_{b=0}^{m-1} \extbinom{n}{b}{d-1}.
\]
One checks that for any $n$ large enough there exists an $m\geq 1$ such that $2m-1\leq n(d-1)$,\, $n(d-1)-(2m-1)< \tfrac12 n(d-1)$ and $m-1< \tfrac12 n(d-1)$. Therefore, with Lemma\nobreakspace \ref {asympt}, we obtain the inequality $\rank(A_{d,n}) \geq 3\cdot d^n - o(d^n)$.
\end{proof}

For computations, the following lemma is useful.

\begin{lemma}\label{balls} For integers $b\geq1$, $n\geq 1$, $d\geq 2$,
\[
\extbinom{n}{b}{d-1} = \sum_{i=0}^{\min(n,\floor{b/d})} (-1)^i \smash{\binom{n}{i}} \smash{\binom{b+n-1-i\cdot d}{n-1}}.
\]
\end{lemma}
\begin{proof}
Let $X \coloneqq \{$ways to put $b$ balls into $n$ containers$\}$ and for $j\in[n]$ let $A_j \coloneqq \{$ways to put $b$ balls in $n$ containers such that container $j$ has at least $d$ balls$\}\subset X$. By the inclusion-exclusion principle \cite[Proposition 1.13]{jukna2011extremal}, the number of elements of $X$ which lie in none of the subsets~$A_j$ is
\[
\sum_{I\subseteq \{1,\ldots,n\}} (-1)^{|I|}\,\Bigl|\bigcap_{j\in I} \!A_j\Bigr|\, = \sum_{I\subseteq \{1,\ldots,n\}} (-1)^{|I|}\,\binom{b+n-1-|I|\cdot d}{n-1}.
\]
Now use that there are $\binom{n}{|I|}$ subsets of size $|I|$ in $\{1,\ldots,n\}$.
The statement about the special case~$d=2$ follows immediately from the definition.
\end{proof}

In the table below we list some values of the lower bound of Proposition\nobreakspace \ref {maincor}.

\begin{center}
\begin{tabular}{p{1.6em}rrrrr}
\toprule
$d$ & 2 & 3 & 4 & 5 & 6\\
\midrule\\[-0.9em]
$n$ &&&&&\\[0.3em]
1 &\bf3 & \bf5 & \bf7 & \bf9 & \bf11 \\
2 &\bf7 & 18 & 33 & 53 & 78 \\
3 &15 & 57 & 142 & 285 & 501 \\
4 &33 & 182 & 601 & 1509 & 3166 \\
5 &68 & 576 & 2507 & 7824 & 19782 \\
6 &141 & 1773 & 10356 & 40329 & 121971 \\
\bottomrule
\end{tabular}
\captionof{table}{Lower bounds for $\rank(A_{d,n})$ from Proposition\nobreakspace \ref {maincor}. The bold numbers are known to be sharp, see Theorem\nobreakspace \ref {charact}.}
\end{center}

\subsection{Rank upper bound}
It is well-known that upper bounds on border rank imply upper bounds on rank. Proposition\nobreakspace \ref {borderrank} implies the following upper bound on $\rank(A_{d,n})$. We will not use the upper bound later, but it provides some context for the lower bound of Proposition\nobreakspace \ref {maincor}.

\begin{proposition}
$\rank(A_{d,n}) \leq (nd+1)d^n$.
\end{proposition}
\begin{proof}
The statement follows from the proof of Theorem 5 in \cite{vrana2013asymptotic}, using that the error degree in the degeneration of the $d$th unit tensor to $A_{d,n}$ is $d$ \cite[Example 15.20]{burgisser1997algebraic}.
\end{proof}

\section{Generalised W-state tensor}\label{secW}

In quantum information theory, the generalised W-state tensor $W_k$ is the tensor in $(\CC^2)^{\otimes k}$ defined by
\[
W_k
\coloneqq e_1\otimes e_0 \otimes \cdots \otimes e_0\, +\, e_0 \otimes e_1\otimes \cdots \otimes e_0\, +\, \cdots\, +\, e_0 \otimes \cdots \otimes e_0 \otimes e_1.
\]
It is not hard to check that, in a particular basis, the tensor of the algebra $A_{2,1} = \CC[x]/(x^2)$ equals~$W_3$. Therefore, $\rank(A_{2,n}) = \rank(W_3^{\otimes n})$ and $\borderrank(A_{2,n}) = \borderrank(W_3^{\otimes n})$. By the following proposition, lower bounds for $\rank(W_3^{\otimes n})$ give lower bounds for~$\rank(W_k^{\otimes n})$.

\begin{proposition}[\cite{chen2010tensor}]\label{induction} $\rank(W_k^{\otimes n}) \geq \rank(W_3^{\otimes n}) + (k-3)(2^n-1)$. 
\end{proposition}

\begin{theorem}\label{wlowerbound} 
\[
\rank(W_k^{\otimes n}) \geq (k-1)2^n + \,\,\max_{m\geq 1}\,\,\,\sum_{b=0}^{2m-2}\binom{n}{b} - 2\sum_{b=0}^{m-1} \binom{n}{b} - (k-3) \geq k\cdot 2^n - o(2^n).
\]
\end{theorem}
\begin{proof}
Combine Proposition\nobreakspace \ref {induction} with Proposition\nobreakspace \ref {maincor} for $A_{2,n}$.
\end{proof}
Chen et al.\ give the lower bound $\rank(W_k^{\otimes n}) \geq (k-1)2^n - k + 2$, which they obtain by combining the lower bound $\rank(A_{2,n}) \geq 2^{n+1} - 1$ with Proposition\nobreakspace \ref {induction}~\cite{chen2010tensor}. Theorem\nobreakspace \ref {wlowerbound} improves the lower bound of Chen et al.
The best upper bound so far is $\rank(W_k^{\otimes n}) \leq (n(k-1) +1)2^n$ \cite{vrana2013asymptotic}.

Below we list some values of the lower bound of Theorem\nobreakspace \ref {wlowerbound}. The first two columns are, in fact, sharp \cite{chen2010tensor}. In Section\nobreakspace \ref {seccube} we will prove the equality $\rank(W_3^{\otimes 3}) = 16$. 
Therefore, the lower bound of Theorem\nobreakspace \ref {wlowerbound} is not sharp in general.

\begin{center}
\begin{tabular}{p{1.7em}rrrrrrrrrrr}
\toprule
$n$ & 1 & 2 & 3 & 4 & 5 & 6 & 7 & 8 & 9 & 10\\
\midrule\\[-0.6em]
$k$ &&&&&&&&&\\[0.4em]
3   & \bf3 & \bf7 & 15 & 33 & 68 & 141 & 297 & 601 & 1230 & 2544 \\
4   & \bf4 & \bf10 & 22 & 48 & 99 & 204 & 424 & 856 & 1741 & 3567 \\
5   & \bf5 & \bf13 & 29 & 63 & 130 & 267 & 551 & 1111 & 2252 & 4590 \\
6   & \bf6 & \bf16 & 36 & 78 & 161 & 330 & 678 & 1366 & 2763 & 5613 \\
7   & \bf7 & \bf19 & 43 & 93 & 192 & 393 & 805 & 1621 & 3274 & 6636 \\
8   & \bf8 & \bf22 & 50 & 108 & 223 & 456 & 932 & 1876 & 3785 & 7659 \\
9   & \bf9 & \bf25 & 57 & 123 & 254 & 519 & 1059 & 2131 & 4296 & 8682 \\
10  & \bf10 & \bf28 & 64 & 138 & 285 & 582 & 1186 & 2386 & 4807 & 9705 \\
\bottomrule
\end{tabular}
\captionof{table}{Lower bounds for $\rank(W_k^{\otimes n})$ from Theorem\nobreakspace \ref {wlowerbound}. The bold numbers are known to be sharp~\cite{chen2010tensor}.}
\end{center}

\section{Gap between rank and border rank}

Our main result Theorem\nobreakspace \ref {mainthintro} follows easily from Theorem\nobreakspace \ref {wlowerbound}.

\begin{proof}[\bf\upshape Proof of Theorem\nobreakspace \ref {mainthintro}] 
By Proposition\nobreakspace \ref {borderrank}, $\borderrank(W_k^{\otimes n}) = 2^n$.
By Theorem\nobreakspace \ref {wlowerbound}, therefore,
\[
\frac{\rank(W_k^{\otimes n})}{\borderrank(W_k^{\otimes n})} \geq k - \frac{o(2^n)}{2^n},
\]
when $n$ goes to infinity.
\end{proof}

\section{Tensor cube of the W-state tensor}\label{seccube}

It is known that the tensor rank of $W \coloneqq W_3$ equals 3 and the tensor rank of the tensor square $W^{\otimes 2}$ equals 7, see~\cite[Lemma~3]{yu2010tensor}. For the tensor cube~$W^{\otimes 3}$ the tensor rank was known to be either~15 or 16 \cite[Theorem~4]{chen2010tensor}. We will prove the following.

\begin{theorem}\label{thmcube}
The tensor rank of $W^{\otimes 3}$ equals $16$.
\end{theorem}

In the following, \emph{algebra} means complex finite-dimensional associative algebra. Let $(V, \phi)$ be an algebra. A subspace $I\subset V$ is called a two-sided ideal if $\phi(I,V) = \phi(V,I) = I$. A two-sided ideal $I$ is called maximal if for all two-sided ideals $J$ with $I \subseteq J \subseteq V$ we have $J = I$ or $J = V$. Similarly for left-ideals. 

\begin{theorem}[Alder-Strassen bound {\cite[Theorem 17.14]{burgisser1997algebraic}}]\label{AlderStrassen}
Let $A$ be an algebra with $t$ maximal two-sided ideals. Then $\rank(A) \geq 2\dim A - t$.
\end{theorem}

\begin{definition}
Let $A$ be an algebra with $t$ maximal two-sided ideals. We say $A$ has \emph{minimal rank} if $\rank(A) = 2\dim A - t$.
\end{definition}

There is a structural description of the algebras of minimal rank \cite{blaser2005complete}. We will only need the following special case. A \emph{simply generated algebra} is an algebra of the form $\CC[x]/(f)$ for some nonconstant polynomial $f\in \CC[x]$. A \emph{generalised null algebra} is an algebra $A$ such that there exist nilpotent elements $w_1,\ldots,w_s\in A$ with $w_iw_j=0$ if $i\neq j$ and $A = \CC[w_1,\ldots,w_s]$. A \emph{local} algebra is an algebra with a unique maximal left-ideal. The \emph{radical} $\rad A$ of $A$ is the intersection of all maximal left-ideals of $A$. The radical is a two-sided nilpotent ideal (see for example \cite{MR674652}).

\begin{theorem}[{\cite[Theorem~17.38]{burgisser1997algebraic}}]\label{charact}
A commutative local algebra is of minimal rank if and only if it is a simply generated algebra or a generalised null algebra.
\end{theorem}

\begin{lemma}[Nakayama's lemma]\label{Nakayama}
Let $A$ be an algebra such that $A/\rad A \cong \CC$. Then~$A$ can be generated as an algebra by $p \coloneqq \dim \rad A/(\rad A)^2$ elements in $\rad A$, that is, there are $w_1,\ldots,w_p\in \rad A$ such that $A = \CC\{w_1,\ldots,w_p\}$. This~$p$ is minimal.
\end{lemma}
We repeat a proof found in \cite{buchi1984uber}.
\begin{proof}
Let $N\coloneqq \rad A$. Let $w_1,\ldots,w_p \in N$ such that $w_1+N^2, \ldots, w_p+N^2$ is a $\CC$-basis for $N/N^2$. One can show by induction that for any $r\geq1$,
\[
\{w_{i_1}\!\cdots w_{i_r} + N^{r+1} \mid 1\leq i_1,\ldots,i_r\leq p\}
\]
generates $N^r/N^{r+1}$ as a vector space over $\CC$. Using that $A/N\cong \CC$ (so $A = \CC \oplus N$) and that $N$ is nilpotent, we get
\[
A = \CC\{w_1,\ldots,w_p\}.
\]

Suppose $p$ is not minimal. Then there is a $q<p$ and a surjective morphism of algebras
\[
\phi : \CC\{X_1,\ldots,X_q\} \twoheadrightarrow A.
\]
Without loss of generality, we may assume that $\phi(X_i)$ is in $N$, since otherwise we can use the decomposition $A=\CC\oplus N$ to map $\phi(X_i)$ to $N$. The set
\[
\{\phi(X_i) + N^2 \mid 1\leq i \leq q \}
\]
is too small to generate $N/N^2$, and $\phi$ maps monomials of degree  $\geq2$ in $\CC\{X_1,\ldots,X_q\}$ to $N^2$. Therefore, $\phi$ is not surjective.
\end{proof}

\begin{proof}[\bfseries\upshape Proof of Theorem\nobreakspace \ref {thmcube}]
The W-state tensor is the structure tensor of the algebra $W\coloneqq\CC[x]/(x^2)$. Consider the algebra 
\[
A\coloneqq W^{\otimes 3} = \CC[x,y,z]/(x^2,y^2,z^2).
\]
It is not hard to see that $A$ is a local algebra with maximal ideal $(x,y,z)$. Let~$N$ be the radical $\rad A = (x,y,z)$. By the Alder-Strassen bound (Theorem\nobreakspace \ref {AlderStrassen}) we have
\[
\rank(A) \geq 2\dim A - 1 = 15.
\]
We will show that $A$ is not of minimal rank. The following type of argument has been used before by Büchi to compute ranks of certain local algebras of dimension at most~5~\cite{buchi1984uber}. Suppose $A$ has minimal rank. By Nakayama's lemma (Lemma\nobreakspace \ref {Nakayama}), the algebra $A$ can be generated by $\dim \rad A/(\rad A)^2 = 3$ elements and no fewer. Therefore, by Theorem\nobreakspace \ref {charact} our algebra $A$ is a generalised null algebra. Hence there are elements~$x_1,x_2,x_3 \in N$ with $x_1x_2=0$, $x_2x_3=0$, and~$x_1x_3=0$ such that $(x_1 + N^2, x_2+N^2, x_3+N^2)$ is a basis of $N/N^2$. On the other hand, $(x+N^2, y+N^2, z+N^2)$ is a basis of $N/N^2$. Therefore, there are elements~$A_{ij}\in \CC$ and $p_i\in N^2$ with
\begin{align*}
x_1 &= A_{11} x + A_{12}y + A_{13}z + p_1,\\
x_2 &= A_{21} x + A_{22}y + A_{23}z + p_2,\\
x_3 &= A_{31} x + A_{32}y + A_{33}z + p_3,
\end{align*}
and $\det A \neq 0$. We may assume that $A_{11}$ is nonzero. We have relations
\begin{align*}
0 = x_1x_2 &= (A_{11}A_{22} + A_{12}A_{21}) xy + (A_{11}A_{23} + A_{13}A_{21})xz\\  &\quad+ (A_{12}A_{23} +  A_{13}A_{22})yz + \textnormal{terms in $N^3$},\\
0 = x_1x_3 &= (A_{11}A_{32} + A_{12}A_{31}) xy + (A_{11}A_{33} + A_{13}A_{31})xz\\  &\quad+ (A_{12}A_{33} +  A_{13}A_{32})yz + \textnormal{terms in $N^3$}.
\end{align*}
Let
\begin{gather*}
f_1 \coloneqq A_{11}A_{22} + A_{12}A_{21},\quad f_2\coloneqq A_{11}A_{23} + A_{13}A_{21},\quad f_3\coloneqq A_{12}A_{23} + A_{13}A_{22},\\
g_1 \coloneqq A_{11}A_{32} + A_{12}A_{31},\quad g_2\coloneqq A_{11}A_{33} + A_{13}A_{31},\quad g_3\coloneqq A_{12}A_{33} +  A_{13}A_{32}.
\end{gather*}
Then we can rewrite the relations as $0 = f_1 = f_2 = f_3 = g_1 = g_2 = g_3.$
With the following Sage code we can compute the syzygy module of the ideal $I \coloneqq (\det(A),f_1, f_2, f_3, g_1, g_2, g_3)\, \CC[A_{ij}]$.

\begin{Verbatim}[framesep=5mm]
R = PolynomialRing(QQ, 3, var_array="a")
A = matrix(R, 3, 3, lambda i,j: "a%d%d" % (i,j))
var("a","b","c")
y = A * vector([a,b,c])
I = R.ideal(det(A),
"a00*a11+a01*a10","a00*a12+a02*a10","a01*a12+a02*a11",
"a00*a21+a01*a20","a00*a22+a02*a20","a01*a22+a02*a21")
L = I.syzygy_module()
print L.str()
\end{Verbatim}

\noindent
One of the syzygies is 
\begin{multline*}
-A_{11}\det(A) =  (A_{13}A_{31} - A_{11}A_{33}) f_1 +   (-3A_{12}A_{31} - A_{11}A_{32}) f_2 + 0 \cdot f_3\\ + 2A_{11}A_{23} g_1 + 2A_{12}A_{21} g_2 + 0 \cdot g_3,
\end{multline*}
implying $\det(A) = 0$, which is a contradiction.
\end{proof}

\paragraph{Acknowledgements} The author is grateful to Matthias Christandl, Markus Bläser and Florian Speelman for helpful discussions. Part of this work was done while the author was visiting the Simons Institute for the Theory of Computing, UC Berkeley and the Workshop on Algebraic Complexity Theory 2015, Saarbrücken. This work is supported by the Netherlands Organisation for Scientific Research (NWO), through the research programme 617.023.116, and by the European Commission, through the SIQS project.

\bibliographystyle{elsarticle-num}
\bibliography{all}

\begin{thebibliography}{10}
\expandafter\ifx\csname url\endcsname\relax
  \def\url#1{\texttt{#1}}\fi
\expandafter\ifx\csname urlprefix\endcsname\relax\def\urlprefix{URL }\fi
\expandafter\ifx\csname href\endcsname\relax
  \def\href#1#2{#2} \def\path#1{#1}\fi

\bibitem{burgisser1997algebraic}
P.~B\"urgisser, M.~Clausen, M.~A. Shokrollahi, Algebraic complexity theory,
  Vol. 315 of Grundlehren der Mathematischen Wissenschaften, Springer-Verlag,
  Berlin, 1997.
\newblock \href {http://dx.doi.org/10.1007/978-3-662-03338-8}
  {\path{doi:10.1007/978-3-662-03338-8}}.

\bibitem{jm2012tensors}
J.~M. Landsberg, Tensors: geometry and applications, Vol. 128 of Graduate
  Studies in Mathematics, American Mathematical Society, Providence, RI, 2012.

\bibitem{haastad1990tensor}
J.~H{\aa}stad, Tensor rank is {NP}-complete, J. Algorithms 11~(4) (1990)
  644--654.
\newblock \href {http://dx.doi.org/10.1016/0196-6774(90)90014-6}
  {\path{doi:10.1016/0196-6774(90)90014-6}}.

\bibitem{hillar2013most}
C.~J. Hillar, L.-H. Lim, Most tensor problems are {NP}-hard, J.~ACM 60~(6)
  (2013) Art. 45, 39.
\newblock \href {http://dx.doi.org/10.1145/2512329}
  {\path{doi:10.1145/2512329}}.

\bibitem{shitov2016hard}
Y.~Shitov, How hard is the tensor rank? (2016).
\newblock \href {http://arxiv.org/abs/1611.01559} {\path{arXiv:1611.01559}}.

\bibitem{hauenstein2013equations}
J.~D. Hauenstein, C.~Ikenmeyer, J.~M. Landsberg, Equations for lower bounds on
  border rank, Exp. Math. 22~(4) (2013) 372--383.
\newblock \href {http://dx.doi.org/10.1080/10586458.2013.825892}
  {\path{doi:10.1080/10586458.2013.825892}}.

\bibitem{MR882307}
V.~Strassen, Relative bilinear complexity and matrix multiplication, J. Reine
  Angew. Math. 375/376 (1987) 406--443.
\newblock \href {http://dx.doi.org/10.1515/crll.1987.375-376.406}
  {\path{doi:10.1515/crll.1987.375-376.406}}.

\bibitem{le2014powers}
F.~Le~Gall, Powers of tensors and fast matrix multiplication, in: I{SSAC}
  2014---{P}roceedings of the 39th {I}nternational {S}ymposium on {S}ymbolic
  and {A}lgebraic {C}omputation, ACM, New York, 2014, pp. 296--303.
\newblock \href {http://dx.doi.org/10.1145/2608628.2608664}
  {\path{doi:10.1145/2608628.2608664}}.

\bibitem{chen2010tensor}
L.~Chen, E.~Chitambar, R.~Duan, Z.~Ji, A.~Winter, Tensor rank and stochastic
  entanglement catalysis for multipartite pure states, Phys. Rev. Lett.
  105~(20) (2010) 200501.
\newblock \href {http://dx.doi.org/10.1103/PhysRevLett.105.200501}
  {\path{doi:10.1103/PhysRevLett.105.200501}}.

\bibitem{blaser2014explicit}
M.~Bl{\"a}ser, Explicit tensors, in: Perspectives in computational complexity,
  Vol.~26 of Progr. Comput. Sci. Appl. Logic, Birkh\"auser/Springer, Cham,
  2014, pp. 117--130.
\newblock \href {http://dx.doi.org/10.1007/978-3-319-05446-9_6}
  {\path{doi:10.1007/978-3-319-05446-9_6}}.

\bibitem{MR2065853}
M.~Bl{\"a}ser, Improvements of the {A}lder-{S}trassen bound: algebras with
  nonzero radical, in: Automata, languages and programming, Vol. 2076 of
  Lecture Notes in Comput. Sci., Springer, Berlin, 2001, pp. 79--91.
\newblock \href {http://dx.doi.org/10.1007/3-540-48224-5_7}
  {\path{doi:10.1007/3-540-48224-5_7}}.

\bibitem{de2008tensor}
V.~de~Silva, L.-H. Lim, Tensor rank and the ill-posedness of the best low-rank
  approximation problem, SIAM J. Matrix Anal. Appl. 30~(3) (2008) 1084--1127.
\newblock \href {http://dx.doi.org/10.1137/06066518X}
  {\path{doi:10.1137/06066518X}}.

\bibitem{allman2013tensor}
E.~S. Allman, P.~D. Jarvis, J.~A. Rhodes, J.~G. Sumner, Tensor rank,
  invariants, inequalities, and applications, SIAM J. Matrix Anal. Appl. 34~(3)
  (2013) 1014--1045.
\newblock \href {http://dx.doi.org/10.1137/120899066}
  {\path{doi:10.1137/120899066}}.

\bibitem{ballico2013stratification}
E.~Ballico, A.~Bernardi, Stratification of the fourth secant variety of
  {V}eronese varieties via the symmetric rank, Adv. Pure Appl. Math. 4~(2)
  (2013) 215--250.
\newblock \href {http://dx.doi.org/10.1515/apam-2013-0015}
  {\path{doi:10.1515/apam-2013-0015}}.

\bibitem{2015arXiv150403732L}
J.~M. Landsberg, M.~Micha{\l}ek, Abelian tensors (2016).
\newblock \href {http://dx.doi.org/10.1016/j.matpur.2016.11.004}
  {\path{doi:10.1016/j.matpur.2016.11.004}}.

\bibitem{blser_et_al:LIPIcs:2016:6434}
M.~Bl{\"a}ser, V.~Lysikov, On degeneration of tensors and algebras, in: 41st
  International Symposium on Mathematical Foundations of Computer Science (MFCS
  2016), Vol.~58, 2016, pp. 19:1--19:11.
\newblock \href {http://dx.doi.org/10.4230/LIPIcs.MFCS.2016.19}
  {\path{doi:10.4230/LIPIcs.MFCS.2016.19}}.

\bibitem{lehmkuhl1989order}
T.~Lehmkuhl, T.~Lickteig, On the order of approximation in approximative
  triadic decompositions of tensors, Theoret. Comput. Sci. 66~(1) (1989) 1--14.
\newblock \href {http://dx.doi.org/10.1016/0304-3975(89)90141-2}
  {\path{doi:10.1016/0304-3975(89)90141-2}}.

\bibitem{blaser2000lower}
M.~Bl{\"a}ser, Lower bounds for the bilinear complexity of associative
  algebras, Comput. Complexity 9~(2) (2000) 73--112.
\newblock \href {http://dx.doi.org/10.1007/PL00001605}
  {\path{doi:10.1007/PL00001605}}.

\bibitem{fahssi2012polynomial}
N.-E. Fahssi, Polynomial triangles revisited (2012).
\newblock \href {http://arxiv.org/abs/1202.0228} {\path{arXiv:1202.0228}}.

\bibitem{flum2006parameterized}
J.~Flum, M.~Grohe, Parameterized complexity theory, Texts in Theoretical
  Computer Science. An EATCS Series, Springer-Verlag, Berlin, 2006.

\bibitem{jukna2011extremal}
S.~Jukna, Extremal combinatorics: with applications in computer science,
  Springer Science \& Business Media, 2011.

\bibitem{vrana2013asymptotic}
P.~Vrana, M.~Christandl, Asymptotic entanglement transformation between {W} and
  {GHZ} states, J. Math. Phys. 56~(2) (2015) 022204, 12.
\newblock \href {http://dx.doi.org/10.1063/1.4908106}
  {\path{doi:10.1063/1.4908106}}.

\bibitem{yu2010tensor}
N.~Yu, E.~Chitambar, C.~Guo, R.~Duan, Tensor rank of the tripartite state {W}
  tensor $n$, Phys. Rev.~A 81~(1) (2010) 014301.
\newblock \href {http://dx.doi.org/10.1103/PhysRevA.81.014301}
  {\path{doi:10.1103/PhysRevA.81.014301}}.

\bibitem{blaser2005complete}
M.~Bl{\"a}ser, A complete characterization of the algebras of minimal bilinear
  complexity, SIAM J.~Comput. 34~(2) (2004/05) 277--298.
\newblock \href {http://dx.doi.org/10.1137/S0097539703438277}
  {\path{doi:10.1137/S0097539703438277}}.

\bibitem{MR674652}
R.~S. Pierce, Associative algebras, Vol.~88 of Graduate Texts in Mathematics,
  Springer-Verlag, New York-Berlin, 1982, studies in the History of Modern
  Science, 9.

\bibitem{buchi1984uber}
W.~B{\"u}chi, {\"U}ber eine klasse von algebren minimalen ranges, Ph.D. thesis,
  Universit{\"a}t Z{\"u}rich (1984).

\end{thebibliography}

\end{document}